\documentclass[10pt,twoside]{siamart1116}

\usepackage[english]{babel}
\usepackage{graphicx,epstopdf,epsfig}
\usepackage{amsfonts,epsfig,fancyhdr,graphics, hyperref,amsmath,amssymb}
\usepackage{tikz}
\usetikzlibrary{arrows,positioning, calc}
\tikzstyle{vertex}=[draw,fill=black!15,circle,minimum size=10pt,inner sep=0pt]

\newcommand{\HE}{}
\newcommand{\DoS}{}
\newcommand{\DoA}{}

\newcommand{\Names}{Bryan Curtis, Jonathan Earl, David Livingston, and Bryan Shader}
\newcommand{\Title}{Resolution of conjectures related to LIGHTS Out!\; and Cartesian Products}

\newtheorem{example}[theorem]{Example}

\renewcommand{\theequation}{\thesection.\arabic{equation}}
\renewtheorem{theorem}{Theorem}[section]

\usepackage{subfigure}

\usepackage{threeparttable}
\usepackage{longtable}
\usepackage{color}
\definecolor{red}{rgb}{1,0,0}

\usepackage{soul} 
\usepackage{xcolor,cancel}

\begin{document}

\bibliographystyle{plain}


\thispagestyle{empty}

 \title{\Title\thanks{Received
 by the editors on \DoS.
 Accepted for publication on \DoA.
 Handling Editor: \HE.}}
\author{ Bryan Curtis, Jonathan Earl, David Livingston, Bryan Shader\thanks{Mathematics Department,
University of Wyoming, Laramie, WY 82071, (bcurtis6@uwyo.edu, jearl5@uwyo.edu, dliving5@uwyo.edu, bshader@uwyo.edu).}}

\markboth{\Names}{\Title}

\maketitle


\begin{abstract}
Lights Out!\ is a game played on a $5 \times 5$ grid of lights, or more generally on a graph. Pressing lights on the grid allows the player to turn off neighboring lights. The goal of the game is to start with a given initial configuration of lit lights and reach a state where all lights are out. Two conjectures posed in a recently published paper about Lights Out!\ on Cartesian products of graphs are resolved.
\end{abstract}


\section{Introduction}
In this short note, we resolve two conjectures from \cite{GPY} concerning Lights Out!\  on Cartesian products of graphs. 

We begin by providing the setting for  the Lights Out!\ problem for a simple graph $G$ with vertices $1,2, \ldots, n$. 
Associated with each vertex of $G$ is a light and a button. If the button at a vertex $i$ is pressed, the lights of the neighbors of $i$ toggle on or off. If a vertex is considered to be a neighbor of itself, this is called {\it closed neighborhood switching}. If not, this is called {\it open neighborhood switching}. Initially, some subset of vertices have their lights on and the complementary set has their lights off. This initial configuration can be represented by the $n \times 1$ vector $b=[b_i]$ where $b_i=1$ if the light at vertex $i$ is initially on, and $b_{j}=0$ if the light at vertex $j$ is initially off. The goal of the Lights Out!\ problem is to press a sequence of buttons so that at the end of the sequence all lights are off. 

First consider the open neighborhood switching, and let $A_G$ be the adjacency matrix of $G$, that is, $A_G=[a_{ij}]$ is the $n \times n$ matrix with $a_{ij}=1$ if $i$ is adjacent to $j$ in $G$, and $a_{ij}=0$ otherwise. Let $e_1, \ldots, e_n$ denote the standard basis vectors. If we start with the configuration corresponding to $b$ and press the button at vertex $i$, then the resulting configuration of lights corresponds to the vector $b+A_Ge_i \bmod 2$. More generally, if we press the  button at vertex $j$ exactly $x_j$ times $(j=1,2,\ldots, n)$, then the resulting configuration of lights corresponds to the vector $b+A_Gx \bmod 2$ where $x=[x_i]$. Note that the ordering in which the buttons are pressed does not matter and 
pressing a button an even number of times is equivalent to not pressing the button at all. Thus,  the Lights Out!\ problem for initial vector $b$, graph $G$ and open neighborhoods  has a solution if and only if the system $A_Gx=b$ has a solution over $\mathbb{Z}_2$.

Let $r(A_G)$ and $\nu(A_G)$ denote the rank and nullity, respectively, of $A_G$ viewed as a matrix over $\mathbb{Z}_2$. Basic facts about linear systems over $\mathbb{Z}_2$ translate into simple facts about the Lights Outs!\ problem on $G$. Namely, the number of initial conditions that can be made to have their lights off is $2^{r(A_G)}$, and for each such initial configuration there are exactly $2^{\nu(A_G)}$ sets of vertices that can be pressed to toggle all lights to the off configuration (see \cite{S,FY}).

Similar statements apply to the Lights Out!\ problem on $G$ for closed neighborhoods; we simply replace $A_G$ by $A_G+I$ throughout. 

%
%
%
%

Let $G$ be a graph with vertices $1$, $2$, \ldots, $m$ and $H$ be a graph with vertices  $1,2,\ldots,n$.
The {\it Cartesian product} of $G$ and $H$ is the graph $G \square H$ with vertex set $\{(i,j): 1\leq i\leq m, 1 \leq j \leq n\}$
such that $(i,j)$ and $(k,\ell)$ are adjacent if and only if $i=k$ and $j$ is adjacent to $\ell$ in $H$, or $i$ is adjacent to $k$ 
in $G$ and $j=\ell$.  The recent paper \cite{GPY} posed conjectures about $\nu(A_{G\square H})$ and $\nu(A_{G\square H}+I)$.
We will prove
 these in Section 3. 

\section{Sylvester's equation}
Let $\mathbb{F}$  be a field, and let $A$ and $B$ be square matrices over $\mathbb{F}$ 
of orders $m$ and $n$ respectively, and $C$ be an $m \times n$ matrix over $\mathbb{F}$. 
The {\it Sylvester equation corresponding to $A$, $B$  and $C$}
is $AX-XB=C$ (e.g. see \cite{BR}).  Throughout, the basic  $p \times p$ Jordan matrix  corresponding to the eigenvalue $\gamma$ is denoted by $J(\gamma, p)$.  

Sylvester's equation arises naturally in the Lights Out!\ setting for Cartesian products of graphs. To see this, let $G$ be a graph on $m$ vertices with adjacency matrix $A$ and let $H$ be a graph on $n$ vertices with adjacency matrix 
$B$. We can view the vertices of $G\square H$ as the positions in an $m \times n$ array; the entry 
in the $(i,j)$ position corresponds to the vertex $i$ of $G$ and $j$ of $H$. As before, the entry $(i,j)$ is 1 if the light is on and 0 otherwise.  Let $E_{ij}$ be the $m \times n$
matrix with a 1 in position $(i,j)$ and $0$'s elsewhere.  Note that the $(k,\ell)$-entry of $AE_{ij}$ is $1$  if and only if
$\ell=j$ and $k$ is adjacent to $i$ in $G$.  Similarly the $(k,\ell)$-entry of $E_{ij}B$ equals $1$ if and only if 
$k=i$ and $\ell$ is adjacent to $j$ in $H$.  Thus, the matrix $AE_{ij} +E_{ij}B$, which is $AE_{ij}-E_{ij}B$ in $\mathbb{Z}_2$, records
the vertices of $G \square H$ that are changed due to pressing cell $(i,j)$ using open neighborhood switching. More generally, the 
configuration $C$ of lights can be turned off using open neighborhood switching if and only if  the system   
$AX-XB=C$ has a solution over $\mathbb{Z}_2$.

Sylvester's equation is well-studied, and in this section we recall some of the known results that will be useful in the 
Lights Out!\ context. It is known that $I_n\otimes A - B^T \otimes I_m$
is a matrix representation of the operator on the vector space $V$ of $m \times n$ matrices over $\mathbb{F}$ that sends $X \in V$ to $AX-XB$  (see \cite[Section 57.4]{B}).
Hence the nullity of $I_n\otimes A - B^T \otimes I_m$ is the dimension of the subspace 
\[ W=\{X \in V: AX=XB\}.\]
Let $\widehat{\mathbb{F}}$ be the algebraic closure of $\mathbb{F}$,  $\widehat{V}$ be the vector space of $m \times n$ matrices over the algebraic closure $\widehat{\mathbb{F}}$ of $\mathbb{F}$, and 
\[ \widehat{W}= \{X \in \widehat{V}: AX=XB\}.\]   While $W$ and 
$\widehat{W}$ are not necessarily equal, their dimensions 
are equal as these sets represent the solution space to a homogeneous system of equations with the same coefficient matrix
but over the field $\mathbb{F}$ and its extension $\widehat{\mathbb{F}}$. Both \cite[Chapter VIII, Section 3]{G} and \cite[Corollary to Theorem 2]{K}  prove the following formula for the  nullity of $I_n\otimes A - B^T \otimes I_m$ over the complexes. Their proofs immediately carry over to any algebraic closed field.

\begin{theorem}
\label{th:sum}
Let  $A$ and $B$ be square matrices over the field $\mathbb{F}$ with 
Jordan Canonical forms  over $\widehat{\mathbb{F}}$ 
\[ \oplus_{i=1}^k J(\lambda_i, m_i) \mbox{ and } 
\oplus_{j=1}^{\ell}  J(\mu_j, n_j),
\]
respectively.  Then 

\[
\nu (I_n\otimes A - B^T \otimes I_m)=  \sum_{i=1}^k \sum_{j=1}^{\ell} \delta_{\lambda_i, \mu_j} \min(m_i,n_j),
\]
where $\delta_{\cdot, \cdot}$ is the Kronecker delta.  
\end{theorem}

Theorem \ref{th:sum} implies that 
\begin{equation}
\label{two} 
\mbox{dim}\;{W}= \sum_{i=1}^k \sum_{j=1}^{\ell} \delta_{\lambda_i, \mu_j} \min(m_i,n_j).
\end{equation}
This formula requires determining the Jordan Canonical Form of both $A$ and $B$ over 
the algebraic closure of $\mathbb{F}$.  In particular, this requires factoring polynomials.  As we see in 
the next section, the formula  is sufficiently strong to resolve the two conjectures in \cite{GKTZ}. 
We end this section by deriving  a similar formula that allows one to do all computations over $\mathbb{F}$ and avoid factorization. We believe this formulation will be more convenient for future research.

First we need to introduce some terminology and recall some classic results which can be found in Chapter 6 of \cite{H}.
  The {\it characteristic matrix}
of $A$ is $xI-A$.  It is known that there exist matrices $U$ and $V$ (whose entries are polynomials in $x$ over $\mathbb{F}$
such that both $\det U$ and $\det V$ are nonzero elements of $\mathbb{F}$) such that $UAV$ has the form 
 \[ S={\rm diag}(s_1(x), s_2(x) ,\ldots, \ldots, s_m(x)), \]
where each  $s_i(x)$ is a monic polynomial and $s_{\ell}(x)$ divides $s_{\ell+1}(x)$
for $k=1, \ldots, m-1$.   The matrix $S$ is the {\it Smith Normal Form} of $xI-A$, 
is unique, and can be determined from $xI-A$ using only the Euclidiean Algorithm for finding gcd's of polynomials. 
The $s_i(x)$ are called the {\it invariant factors of $xI-A$}. Let $p_1(x), \ldots, p_u(x)$ be the distinct irreducible factors of  the characteristic polynomial of $c_A(x)=s_1(x)s_{2}(x) \cdots s_m(x)$.
Then there exist nonnegative integers $e_i(j)$ such that 
\[ s_i(x)= p_1(x)^{e_1(i)} {p_2(x)^{e_2(i)}} \cdots p_u(x)^{e_u(i)} \quad \mbox{for $i=1, \ldots, m$.}
\]

Now we do the same for the $n \times n$ matrix $B$.  Take the Smith Normal Form of 
$B$ to be 
\[ T={\rm diag}(t_1(x), t_{2} (x), \ldots, t_n(x)), \]
where each of the $t_i(x)$ is a  monic polynomial and $t_{\ell}(x)$ divides $t_{\ell+1}(x)$
for $\ell=1, \ldots, n-1$. 
Let $q_1(x), \ldots, q_v(x)$ be the distinct irreducible factors of $c_B(x)=t_{1} (x)t_{2}(x) \cdots t_n(x)$.
Then there exist nonnegative integers $f_i(j)$ such that 
\[
t_i(x)= q_1(x)^{f_1(i)} q_2(x)^{f_2(i)} \cdots q_t(x)^{f_v(i)} \quad \mbox{
for $i=1, \ldots, n$.}
\]

If $\mathbb{F}$ is separable, which is the case when $\mathbb{F}$ is finite or  of characteristic 0,  and $\lambda$
is an eigenvalue of $A$ over $\widehat{F}$, then the sizes of the Jordan blocks in the Jordan Canonical Form of $A$
corresponding to $\lambda$ are the nonzero $e_j(i)$ ($i=1,\ldots, m$), where  $j$ is the unique 
index such that $\lambda$ is a root of $p_j(x)$. Similarly, if $\mu$ is an eigenvalue of $B$, then  the sizes of the Jordan blocks in the Jordan Canonical Form of $B$
corresponding to $\mu$ are the nonzero $f_j(i)$ ($i=1, \ldots, n$) for which 
$\mu$ is a root of $q_j(x)$.

\begin{corollary}
\label{SNF}
Let $A$ and $B$ be $m \times m$ and $n\times n$ matrices, respectively, over the separable field $\mathbb{F}$
with invariant factors $s_1(x), \ldots, s_m(x)$, and $t_{1}(x), \ldots, t_{n}(x)$, respectively. 
Then the nullity of $I\otimes A - B^T \otimes I$, is given by 
\[
\sum_{i=1}^m\sum_{j=1}^n {\rm deg} ( {\rm gcd}( s_i(x),t_j(x)).
\]
\end{corollary}

\begin{proof}
Let $\Lambda_k$ be the set of roots of $p_k(x)$ over $\widehat{\mathbb{F}}$, 
$\Gamma_{\ell}$ be the set of roots of $q_{\ell}(x)$ over $\widehat{\mathbb{F}}$, $\Lambda= \cup_{k=1}^m \Lambda_k$ and 
$\Gamma= \cup_{\ell=1}^n \Gamma_{\ell}$. Then
\begin{eqnarray*}
\sum_{i=1}^m\sum_{j=1}^n \mbox{\rm deg}(\mbox{\rm gcd}(s_i(x), t_j(x))) &=&
{\sum_{i=1}^m\sum_{j=1}^n \mbox{\rm deg}(\mbox{\rm gcd}
(
 \prod_{k=1}^u p_k(x)^{e_i(k)}, 
 \prod_{\ell=1}^v q_{\ell}(x)^{f_j(\ell)}
)}
\\
& = &
{\sum_{i=1}^m\sum_{j=1}^n  \sum_{k=1}^u \sum_{\ell=1}^v
\mbox{\rm deg}(\mbox{\rm gcd} (p_k(x)^{e_i(k)}, q_{\ell}(x)^{f_j(\ell)}
)}\\
&= & 
\sum_{i=1}^m\sum_{j=1}^n  \sum_{k=1}^u \sum_{\ell=1}^v
\delta_{p_k(x),q_{\ell}(x)} \mbox{\rm deg} (p_k(x))\min (e_i(k),f_j(\ell))
\\ 
&= &
\sum_{i=1}^m\sum_{j=1}^n  \sum_{k=1}^u \sum_{\ell=1}^v
\sum_{\lambda \in \Lambda_k \cap \Gamma_{\ell}}  \min (e_i(k),f_j(\ell))\\
&=& 
 \sum_{k=1}^u 
\sum_{\ell=1}^v 
\sum_{\lambda  \in \Lambda_k}
\sum_{\mu \in \Gamma_{\ell}} 
\delta_{\lambda, \mu}  \sum_{i=1}^m\sum_{j=1}^n 
 \min (e_i(k),f_j(\ell)).
\end{eqnarray*}
Noting that the $e_i(k)$ and $f_j(\ell)$ 
are the sizes of the Jordan blocks of $A$, respectively $B$, corresponding to the eigenvalues of $\lambda$ and $\mu$ 
respectively, the result follows from Theorem \ref{th:sum}.

\end{proof}

\begin{corollary}
\label{SNF3}
Let $A$  be an $m \times m$ matrix over the separable field $\mathbb{F}$
with invariant factors $s_1(x)$, \ldots, $s_m(x)$.
Then the nullity of $I\otimes A - A^T \otimes I$ is given by 
\[
\sum_{i=1}^m (2m-2i+1) \mbox{\rm deg}(s_i(x))
\]
\end{corollary}

\begin{proof}
Since $s_i(x)$ divides {$s_{i+1}(x)$} for {$i=1,\ldots, m-1$}, 
$\mbox{gcd}(s_i(x),s_j(x))=s_i(x)$ when {$i\leq j$}. 
The result now follows from Corollary \ref{SNF}.
\end{proof}

If $A$ is non-derogatory (that is, its minimal and characteristic polynomial are equal), 
then it has only one invariant factor not equal to $1$, and hence the formula in Corollary \ref{SNF}
simplifies to 
\[
\sum_{j=1}^n \mbox{deg} (\mbox{gcd} (c_{A}(x), t_j(x)).\]
 In particular, we have the 
following result for Cartesian products involving a path $P_m$.
 
\begin{corollary}
\label{SNF2}
Let $G$ be a graph and let $(s_1, \ldots, s_n)$ be the invariant factors of $xI-A_G$ over $\mathbb{Z}_2$.
Then the nullity of the adjacency matrix of $P_m \square G$ equals
\[ {\sum_{i=1}^n \mbox{\rm deg}( \mbox{\rm gcd}(c_{P_m}(x), s_i(x)).}
\]
\end{corollary}

\begin{proof}
The submatrix obtained from $A_{P_m}$ by deleting its first row and last column has 
determinant $1$. Hence the geometric multiplicity of each eigenvalue of $A_{P_m}$ is $1$.
This implies that the Jordan Canonical Form of $A_{P_m}$ has exactly one block for each 
eigenvalue, and hence $A_{P_m}$ is non-derogatory.
The result now follows from Corollary \ref{SNF}.
\end{proof}

We conclude {this section} with a few simple examples. 

\begin{example}
Consider an $n \times n$ matrix $A$ and an $m \times m$ matrix $B$, each nilpotent over $\mathbb{F}=\mathbb{Z}_2$ and with 
Jordan Canonical Forms 
\[
\big(\oplus_{i=1}^{n-3} J(0, 1)\big)\oplus J(0,3) \mbox{ and } 
\big(\oplus_{j=1}^{m-3}  J(0, 1)\big)\oplus J(0,3),
\]
respectively.  By Corollary $\ref{SNF}$ the nullity of $I \otimes A - B^T \otimes I$ over $\mathbb{Z}_2$ is given by

\[
\sum_{i=1}^{n-3}\sum_{j=1}^{m-3}\min(1,1)+\sum_{j=1}^{m-3}\min(3,1)+\sum_{i=1}^{n-3}\min(1,3)+\min(3,3)=(m-2)(n-2)+2.
\]
It can be verified that the adjacency matrix of the star graph $S_n$ (as illustrated below in Figure $1$) on $n$ vertices, with $n$ odd, is of the form described above. Thus $\nu(S_n \square S_m)=(m-2)(n-2)+2$ for $n,m$ odd.




\end{example}

\begin{example}
Consider $S_n \square P_m$, and let $\mathbb{F}=\mathbb{Z}_2$. Let $\nu=\nu(A_{P_m})$. Then by Corollary $\ref{SNF2}$,

\[
\nu(S_n\square P_m)= \left\{ \begin{array}{ll}
 0 & \mbox{ if $\nu =0$,}\\
  (m-3)+\nu& \mbox{ if $1\leq \nu  \leq 3$,}  and  \\
 m & \mbox{ if $3<\nu$.} \end{array} \right.
 \]

\end{example}

\begin{center}
\begin{tikzpicture}[level/.style={sibling distance=20mm/#1}]
\node [vertex]{}
  child {
    node [vertex]  {}
  }
  child {
    node [vertex]  {}
  };
\end{tikzpicture} \qquad
\begin{tikzpicture}[level/.style={sibling distance=20mm/#1}]
\node [vertex]{}
  child {
    node [vertex]  {}
  }
  child {
    node [vertex]  {}
  }
  child {
    node [vertex]  {}
  }
   child {
    node [vertex]  {}
  };
\end{tikzpicture}\\
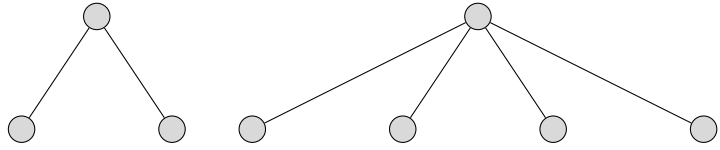
{Figure 1: The star graphs $S_3$ and $S_5$}
\end{center}


\begin{example}  
Let $H$ be the Petersen graph.  The Smith Normal Form of $xI-A_H$ over $\mathbb{Z}_2$ 
is 
\[ {S}={{\rm diag}}(1,1,1,1,1,1,(x+1),(x+1)x,(x+1)x,(x+1)x,(x+1)^2x).
\]
Hence, by Corollary $\ref{SNF}$, the nullity of {$A_H\square A_H$} is 
\[ \sum_{i=1}^{10} \mbox{\rm deg}({{\rm gcd}}(s_i,s_j)) =42.
\]
\end{example}

\section{Proof of conjectures}
Our proof depends on the following simple result about partitions, and Theorem \ref{th:sum}. A {\it partition of the nonnegative integer $r$} is a tuple $\pi=(\pi_1, \ldots, \pi_k)$ of positive 
integers  with $r=\pi_1+ \cdots + \pi_k$. 

\begin{lemma}
\label{partition}
Let $\pi=(\pi_1, \ldots, \pi_k)$ and $\tau=(\tau_1, \ldots, \tau_{\ell})$ be partitions of $r$ and $s$ respectively. Then 
\begin{equation}
\label{one} 
\sum_{i=1}^k \sum_{j=1}^{\ell} \min(\pi_i, \tau_j) \geq \min(r,s).
\end{equation}
\end{lemma}
\begin{proof}
Without loss of generality we may assume that $s\geq r$.
Consider an $i$.  If there exists a $j$ such that $\tau_j \geq \pi_i $, then 
$\sum_{j=1}^{\ell} \mbox{min}(\pi_i, \tau_j) \geq \pi_i$. Otherwise,
\[ \sum_{j=1}^{\ell} \mbox{min}(\pi_i, \tau_j) = \sum_{j=1}^{\ell} \tau_j= s\geq r \geq \pi_i. \]
Hence for each $i$, we have $\sum_{j=1}^{\ell} \mbox{min}(\pi_i, \tau_j) \geq \pi_i$.
The result now follows by summing these inequalities over $i=1,\ldots, k$.
 \end{proof}

We note, but don't make use of, the fact that equality holds in (\ref{one}) if and only if 
 $k=1$ or $\ell=1$, and $r=s$.

Let $\lambda$ be an eigenvalue of $A$, and let $S_{\lambda}= \{(j,k): \lambda_j=\lambda=\mu_{k} \}$.
Let $\alpha_{A}(\lambda)$ and $\alpha_{B}(\lambda)$ be the algebraic multiplicity of $\lambda$ as an eigenvalue of
$A$ and $B$ respectively.  Note that  the contribution to $\mbox{dim}\; W$ in (\ref{two}) corresponding 
to $\lambda$ is given by 
\[
\sum_{ (j,k) \in S_{\lambda}} \min(m_j, n_k) .
\]
By Lemma \ref{partition}, this contribution is at least $\min(\alpha_{A}(\lambda), \beta_B(\lambda))$.  This last quantity 
is the multiplicity of $\lambda$ as a root of $\mbox{gcd}(c_{A}(x), c_{B}(x))$. 
Hence we have proven the following result. 

\begin{corollary}
\label{main}
Let $A$ and $B$ be square matrices over the field $\mathbb{F}$ of order $m$ and $n$ respectively. 
Then the nullity of {$I_n \otimes A - B^T \otimes I_m$} is at least the degree of the greatest common 
divisor of $c_A(x)$ and $c_B(x)$.
\end{corollary}
 
The following  corollaries prove Conjectures  4.1 and 4.2 of \cite{GPY}.
\begin{corollary}
Let $G$ and $H$ be graphs with adjacency matrices $A$ and $B$ respectively,  and $\mathbb{F}$ be a field of characteristic $2$.
Then the nullity of  $A_{G \square H}$ is at least the degree of the 
greatest common divisor of $c_A(x)$ and $c_B(x)$ over $\mathbb{F}$.
\end{corollary}

\begin{proof}
Note that the adjacency matrix of $G \square H$  is $A\otimes I + I\otimes B$, 
which is {$I\otimes B - A^T\otimes I$} since $\mathbb{F}$ has characteristic $2$ {and $A$ is symmetric}.
The result now follows from Corollary \ref{main}.
\end{proof}

\begin{corollary}
Let $G$ and $H$ be graphs with adjacency matrices $A$ and $B$ respectively,  and $\mathbb{F}$ be a field of characteristic $2$.
Then the nullity of  $A_{G \square H}+I$ is at least the degree of the 
greatest common divisor of $c_A(x+1)$ and $c_B(x)$ over $\mathbb{F}$.
\end{corollary}

\begin{proof}
Note that  
 $A_{G \square H}+I$  is $(A+I)\otimes I + I\otimes B$, 
which is {$I\otimes B - (A+I)^T\otimes I$} since $\mathbb{F}$ has characteristic $2$ {and $A + I$ is symmetric}.
The result follows from Theorem \ref{main} and the observation that 
$c_{A+I}(x)= c_{A}(x+1)$.
\end{proof}

\bibliographystyle{plain}

\end{document}